\documentclass[12pt,a4paper]{amsart}
\usepackage{times,amsmath,url,amscd}
\usepackage{amssymb}
\usepackage{layout}
\usepackage{amsthm}

\numberwithin{equation}{section}

\newtheorem{thm}[equation]{Theorem}
\newtheorem{prop}[equation]{Proposition}

\newtheorem{lemma}[equation]{Lemma}

\theoremstyle{definition}
\newtheorem{defn}[equation]{Definition}
\newtheorem{ex}[equation]{Example}
\newtheorem{exs}[equation]{Examples}
\newtheorem{tabel}[equation]{Table}

\theoremstyle{remark}

\DeclareMathOperator{\Hom}{Hom}

 \DeclareMathOperator{\Md}{Mod}
 \DeclareMathOperator{\Comd}{Comod}

\DeclareMathOperator{\Dsc}{Disc} \DeclareMathOperator{\lfg}{LFG}

 \DeclareMathOperator{\Ann}{Ann}

\newcommand{\ie}{\textit{i.e.}}
\newcommand{\eg}{\textit{e.g.}}

\newcommand{\N}{\mathbb{N}}
\newcommand{\Z}{\mathbb{Z}}
\newcommand{\Zpl}{\ensuremath{\Z_{(p)}}}
\newcommand{\Q}{\mathbb{Q}}

\renewcommand{\phi}{\varphi}

\newcommand{\from}{\colon}

\renewcommand{\geq}{\geqslant}
\renewcommand{\leq}{\leqslant}
\let\iso\cong
\renewcommand{\ss}{\subseteq}
\newcommand{\ssne}{\subset}
\let\oldmod\mod
\renewcommand{\mod}{\!\!\oldmod}
\newcommand{\ab}[1]{\langle #1 \rangle}
\newcommand{\floor}[1]{\lfloor #1 \rfloor}

\newcommand{\Ds}{\bigoplus}
\newcommand{\tp}{\otimes}

\newcommand{\Mod}[2]{{}_{#1}\!\Md_{#2}}
\newcommand{\Comod}[2]{{}_{#1}\!\Comd\!_{#2}}
\newcommand{\Disc}[2]{{}_{#1}\!\Dsc\!_{#2}}
\newcommand{\LFG}[2]{{}_{#1}\!\lfg\!_{#2}}

\newcommand{\Zp}{\Z_{(p)}}

\let\smash\wedge
\renewcommand{\wedge}{\vee}

\begin{document}

\title[Discrete module categories and operations in $K$-theory]{Discrete module categories and operations in $K$-theory}

\author{A.J.~Hignett}
\address{Metaswitch Networks, Causewayside House, 160 Causewayside, 
Edinburgh EH9 1PR, UK.
}
\email{tony.hignett@metaswitch.com}

\author{Sarah Whitehouse}
\address{School of Mathematics and Statistics, University of Sheffield, Sheffield S3 7RH, UK.}
\email{s.whitehouse@sheffield.ac.uk}

\begin{abstract}
We study the categories of discrete modules for topological rings arising as the 
rings of operations in various kinds of topological $K$-theory. We prove that for these rings
the discrete modules coincide with those modules which are locally finitely generated over the ground ring.
\end{abstract}

  
\date{$14^{\text{th}}$ October 2010}
\maketitle

\section{Introduction}
\label{SecIntro}

For well-behaved coalgebras $C$, the category of $C$-comodules is isomorphic to the
category of discrete modules over the dual algebra $A=C^*$. It is thus interesting to
be able to characterize discrete modules.
In~\cite{ccw3} it was shown that the category of discrete modules over the algebra
of operations in $p$-local $K$-theory is isomorphic to the category of locally
finitely generated modules. In this paper we will show that the same is true
for a family of related examples. The method used also gives a new and simpler proof
in the original case.

We work over a commutative unital Noetherian ground ring $R$; 
in the applications $R$ will be the $p$-local integers $\Zpl$ for some prime $p$;
we also consider $R=\Z$.
We are interested in topological algebras $A$ over $R$
of the following form. There is an infinite family of elements $\{a_n\,|\,n\geq 0\}$
of $A$, a \emph{topological basis}, such that $A$ consists of precisely the
infinite sums $\{\sum_{n\geq 0}r_na_n\,|\,r_n\in R\}$. 
The sets $A_m=\{\sum_{n\geq m}r_na_n\,|\,r_n\in R\}$ are ideals of $A$
and $A$ is complete with
respect to the filtration by these.

A simple example is provided by the power series ring $R[\![t]\!]$ with the
usual $t$-adic filtration. This example has many special features, of course, as
the filtration ideals are principal and $(t^r)(t^s)=(t^{r+s})$. 
These features lead to nice properties of the category of discrete modules over a
power series ring.
In contrast the examples
we consider are much more complicated and the filtrations are not multiplicative.
Nonetheless we establish conditions on $A$ under which some of these good properties 
are preserved. The examples arise as the linear duals of coalgebras having a particular form
and we also formulate our conditions in terms of the coalgebra to which $A$ is dual.

This article is organized as follows.
In Section~\ref{SecDiscreteMods} we define discrete modules and discuss some basic properties.
Section~\ref{SecComods} is devoted to comodules over a coalgebra $C$. In Proposition~\ref{PropDiscIsoComod}
we give conditions on $C$ under which the category of right $C$-comodules is isomorphic
to the category of discrete left modules over the dual of $C$. This explains our interest
in discrete modules.
In Section~\ref{SecRegCoalg} we introduce the notion of a regular coalgebra
and provide a characterization of the units in the dual of such a coalgebra.
In Section~\ref{SecKCoops} we explain how examples of topological interest fit into
this framework.

We define the notion of a module which is locally finitely generated over the ground ring 
in Section~\ref{SecLFG} and here we give our
main technical result, Theorem~\ref{ThmLFGIsDiscrete}, providing conditions on $A$ under which 
such locally finitely generated modules coincide with discrete modules. 
The conditions can be interpreted as requiring 
$A$ to have a certain resemblance to a power series ring.
We also discuss a kind of dual result, by providing the corresponding conditions on
the coalgebra to which $A$ is dual (Theorem~\ref{ThmLFGIsDiscrete(C)}).
In Section~\ref{SecApplications} we do the further work necessary to show 
that the main technical results
apply to various examples of topological interest. These results are collected together
in Theorem~\ref{ThmLFGForKs}.
\smallskip

This paper is based on work in the Ph.D. thesis of the first author~\cite{hignett},
produced under the supervision of the second author.

\section{Discrete modules}
\label{SecDiscreteMods}

We work over a commutative unital Noetherian ground ring $R$. In the applications $R$ will be
$\Zpl$;
we also consider $R=\Z$. All modules considered will be left modules and all tensor products will be over $R$.
Let $A$ be a unital topological algebra over $R$ and write $\Mod{A}{}$ for the category
of $A$-modules. In this section we define and 
investigate discrete $A$-modules and the subcategory of $\Mod{A}{}$ they form.

Our interest in discrete modules arises from the fact that,
for sufficiently nice cohomology theories $E$, the $E$-homology of a space 
or spectrum is a 
discrete module over the algebra of operations of the theory.
The discrete module category of the topological $\Zp$-algebra 
of degree zero stable operations in $p$-local $K$-theory was studied in~\cite{ccw3}. 
This paper is motivated by the desire to extend the results of~\cite{ccw3} to 
many related examples of topological interest.

\begin{defn}
Let $A$ be a topological $R$-algebra.
An $A$-module $M$ is \emph{discrete} if it is continuous when given the 
discrete topology, \ie\ if the action map $A \times M \to M$ is continuous.

The category $\Disc{A}{}$ of discrete $A$-modules is the full subcategory 
of $\Mod{A}{}$ with objects the discrete $A$-modules. 
\end{defn}

Of course, $\Disc{A}{}$ depends on the topology on $A$,
but we leave it out of the notation. In general, $A$ will not itself be a 
discrete $A$-module.

It is straightforward to check that the category $\Disc{A}{}$ is closed 
under submodules, quotients and direct sums. Hence it is a cocomplete abelian category.

We will consider those $A$ whose topology is given by a filtration
\[
A=A_0 \supset A_1 \supset A_2 \supset \dots
\]
by (two-sided) ideals. We recall that such a filtration is
\emph{Hausdorff} if the map $A \to \varprojlim A/A_n$ is injective (\ie\ $\bigcap_{n \geq 0} A_n$ is zero) and
\emph{complete} if $A \to \varprojlim A/A_n$ is surjective.

The algebras we work with will be \emph{pro-finitely generated} as modules over the ground ring
in the sense of the following definition.

\begin{defn}\label{Defnprofg}
A topological $R$-algebra $A$, with topology given by a filtration
\[
A=A_0 \supset A_1 \supset A_2 \supset \dots,
\]
is \emph{pro-finitely generated} if it is complete and Hausdorff, and each 
$A/A_n$ is a finitely generated $R$-module.
\end{defn}

Such an $R$-algebra is an inverse limit of finitely generated $R$-modules.
\medskip

In practice we characterize discrete modules by the following property.

\begin{lemma}\label{LemmaCharDisc}
\cite[2.4]{ccw3}. An $A$-module $M$ is discrete if and only if 
for every $x \in M$ there is $n$ with $A_nx=0$.
\qed
\end{lemma}

For example, if $I$ is an ideal of $A$, then $A/I$ is discrete if and only if there is some $n$ such that $A_n \ss I$.
\medskip

From now on, we assume that the filtration of $A$ is complete and Hausdorff and 
that each containment $A_n \ssne A_{n-1}$ is strict. The canonical examples of 
discrete modules are the quotients $A/A_n$
for $n \geq 0$. Note that, if $M$ is discrete, there is not necessarily an $n$ 
such that $A_nM=0$: consider $\Ds_{n \geq 0} A/A_n$.
\smallskip

It is important to note that we do not assume that our filtration is multiplicative.
The examples we are interested in arise naturally in topology; they are complete with
respect to their filtration topologies, but they are not complete with respect to any
multiplicative filtration. This feature makes them complicated rings and leads to
possibly unexpected properties of the discrete module categories. 

For example, discrete module categories are in general not closed under extensions 
in the corresponding module category. Take $m,n \geq 0$ and suppose that there is no $k$ such that
\[
A_k \ss A_mA_n.
\]
Then the exact sequence
\[
0 \to A_n/A_mA_n \to A/A_mA_n \to A/A_n \to 0,
\]
of $A$-modules, where the maps are the natural inclusion and projection, 
shows that $A/A_mA_n$ is a non-discrete extension of discrete modules.

In contrast, the prototypical example with a multiplicative filtration
is the power series ring $\Z[\![t]\!]$ with its usual $t$-adic filtration.
It is easy to check that the category $\Disc{\Z[\![t]\!]}{}$ is closed under extensions in $\Mod{\Z[\![t]\!]}{}$.

\begin{defn} 
We define a functor
    \[
    \Dsc \from \Mod{A}{} \to \Disc{A}{}
    \]
as follows. On objects, for $M$ an $A$-module, define
\[
\Dsc M = \{ x \in M\,|\,A_nx=0 \textup{ for some $n$} \}\subseteq M,
\]
the maximal discrete submodule of $M$. On morphisms, 
for  an $A$-linear map $f:M\to N$,
define
$\Dsc f$ by restricting the domain of $f$ to $\Dsc M$
and the codomain to $\Dsc N$.
(If $x \in \Dsc M$ and $A_nx=0$ then $A_nf(x) = f(A_nx) = 0$, so $f|_{\Dsc M}$ lands in $\Dsc N$.)
\end{defn}

It is straightforward to check that 
the functor $\Dsc$ is right adjoint to the inclusion of $\Disc{A}{}$ into $\Mod{A}{}$.

\section{Comodules}
\label{SecComods}

In this section we recall some results about coalgebras and comodules. Our main reference is~\cite{C&C}.
Let $C$ be an $R$-coalgebra, with comultiplication map $\Delta:C\to C\otimes C$. We write $\Comod{}{C}$ for the category of right
$C$-comodules and comodule maps. If the coalgebra $C$ is flat over the ground ring
then $C$-comodules have good finiteness properties.

\begin{prop}\label{PropCoFiniteness}
If $C$ is $R$-flat, then
\begin{enumerate}
\item \cite[3.16]{C&C} every element of a right $C$-comodule $M$ is contained in a subcomodule which is finitely generated as an $R$-module;
\item every element of $C$ is contained in a subcoalgebra of $C$ which is finitely generated as an $R$-module.\qed
\end{enumerate}
\end{prop}

The algebras we are interested in arise as the duals of such coalgebras.
We define the $R$-module $C^* = \Mod{R}{}(C,R)$,
the $R$-linear dual. Denote the evaluation pairing $C^* \tp C \to R$ by $\ab{-,-}$ 
and recall that if $f \from M \to N$ is a map of $R$-modules, the dual map $f^* \from N^* \to M^*$ is defined by
\[
\ab{f^*(\phi),m} = \ab{\phi,f(m)}
\]
for $\phi \in N^*$ and $m \in M$. Then $C^*$ has the convolution product $*$: 
if $a_1,a_2 \in C^*$, $c \in C$ and $\Delta(c) = \sum c_i^{(1)} \tp c_i^{(2)}$,
\[
\ab{a_1 * a_2,c} = \sum \ab{a_1,c_i^{(1)}} \ab{a_2,c_i^{(2)}}.
\]
This makes $C^*$ into an $R$-algebra.

\begin{lemma}\label{LemmaDualAlgebraProfinite}
Let $C$ be a flat $R$-coalgebra. Then the dual $A=C^*$ of $C$ is a filtered $R$-algebra
which is pro-finitely generated in the sense of Definition~\ref{Defnprofg}.
\end{lemma}

\begin{proof}
Because $C$ is flat, by Proposition~\ref{PropCoFiniteness}, we may write
\[
C = \bigcup_{\alpha \in I} C_\alpha,
\]
where each $C_\alpha$ is an $R$-finitely generated subcoalgebra of $C$. 
Since $R$ is Noetherian, the algebra $C_\alpha^*$ is also $R$-finitely 
generated and, because the dual functor $(-)^*$ takes colimits 
(in this case, unions) to limits,
\[
A \iso \varprojlim C_\alpha^*.
\]
Let $A_\alpha$ be the kernel of the restriction map $A=C^*\to C^*_\alpha$. Then $A$ is filtered
by the ideals $A_\alpha$ and $A/A_\alpha\cong C_\alpha^*$.
So $A$ is pro-finitely generated.
\end{proof}

Now let $M$ be a right $C$-comodule $M$, 
with coaction map $\rho_M:M\to M\otimes C$.
Define a left $A$-action on $M$ by
\begin{equation}\label{EqComodIsMod}
ax = \sum_i \ab{a, c_i} x_i,
\end{equation}
where $a \in A$, $x \in M$ and $\rho_M(x) = \sum x_i \tp c_i$.

This makes $M$ a left $A$-module and
if $f \from M \to N$ is a map of right $C$-comodules, then $f$ is also a map of left $A$-modules when 
$M$ and $N$ are given the above action.

In particular, $C$ itself is an $A$-module and this action coincides with
 the usual action on a coalgebra by
its dual. We define a functor
\[
i \from \Comod{}{C} \to \Mod{A}{},
\]
where $i(M)$ is $M$ with the above $A$-action and $i(f) = f$. Note that $i$ is faithful. In general, 
it does not have other nice properties, unless we impose conditions on $C$.

Recall that $C$ is a subgenerator of $\Comod{}{C}$, that is, every $C$-comodule is a 
subcomodule of a quotient of a direct sum of copies of $C$. If $M$ is an $A$-module, 
define the category $\sigma[M]$ to be the full subcategory of $\Mod{A}{}$ subgenerated by $M$.

\begin{thm}\label{ThmComodsAsMods}
\cite[4.3]{C&C}. The following are equivalent.
\begin{enumerate}
\item $i(\Comod{}{C})$ is a full subcategory of $\Mod{A}{}$.
\item $i(\Comod{}{C}) = \sigma[i(C)]$.
\item $C$ is a locally projective $R$-module.\qed
\end{enumerate}
\end{thm}

We refer to~\cite[42.9]{C&C} for the definition of a locally projective module.
(This condition on modules lies between projectivity and flatness.)

From Theorem~\ref{ThmComodsAsMods} we see that, if $C$ is locally projective as an $R$-module (and hence flat),
\begin{enumerate}
\item $i$ is a inclusion of categories
\[
\Comod{}{C} \to \Mod{A}{};
\]
\item $\Comod{}{C}$ has all kernels, cokernels and direct sums;
\item every element of a $C$-comodule is contained in an $R$-finitely generated $C$-comodule;
\item $C$ may be written as a union $\bigcup C_\alpha$ of finitely generated subcoalgebras $C_\alpha$.
\end{enumerate}

It is also known when $i$ is an equality.

\begin{prop}\label{PropComodIsoMod}
\cite[4.7]{C&C}.
Suppose $C$ is locally projective. Then $i(\Comod{}{C})$ and $\Mod{A}{}$ 
are equal if and only if $C$ is a finitely generated and projective $R$-module.\qed
\end{prop}

\begin{lemma}\label{LemmaComodsAreDiscrete}
Let $C$ be an $R$-coalgebra with dual algebra $A=C^*$.
Suppose that $C = \bigcup C_\alpha$ for finitely generated subcoalgebras $C_\alpha$.
Then $i(C)$ is a discrete $A$-module. If in addition $C$ is a locally projective $R$-module,
then $i(\Comod{}{C})$ is a full subcategory of $\Disc{A}{}$.
\end{lemma}
 
\begin{proof}
If $c \in C$, take $\alpha$ such that $c \in C_\alpha$. 
Then $\Delta(c) = \sum c_i^{(1)} \tp c_i^{(2)} \in C_\alpha \tp C_\alpha$. If $a \in A$,
\[
ac = \sum \ab{a,c_i^{(2)}}c_i^{(1)},
\]
so $A_\alpha c = 0$. Hence $i(C)$ is a discrete $A$-module.

If $C$ is a locally projective $R$-module, then by Theorem~\ref{ThmComodsAsMods} 
$i(\Comod{}{C})$ is the full subcategory $\sigma[i(C)]$ of $\Mod{A}{}$. 
Since $\Disc{A}{}$ has direct sums, kernels and cokernels, 
 $i(\Comod{}{C})$ is a full subcategory of $\Disc{A}{}$.
\end{proof}

\begin{prop}\label{PropDiscIsoComod}
Let $C$ be a locally projective $R$-coalgebra with dual algebra $A$. 
Suppose that $C = \bigcup C_\alpha$ for subcoalgebras $C_\alpha$ which are 
finitely generated and projective $R$-modules. Then $i$ is an isomorphism of categories
\[
\Comod{}{C} \overset{\iso}{\to} \Disc{A}{}.
\]
\end{prop}

\begin{proof}
We have already observed that in this case $i(\Comod{}{C})$ is a full 
subcategory of $\Disc{A}{}$. Directly, if $M$ is a right $C$-comodule, 
$x \in M$ and $\rho_M(x) = \sum x_i \tp c_i$, we take $\alpha$ such that 
all $c_i$ are in $C_\alpha$; then $A_\alpha x = 0$.

Conversely, take a discrete left $A$-module $N$ and an $\alpha$. Let
\[
N_\alpha = \{ y \in N\,|\,A_\alpha y = 0 \};
\]
then $N_\alpha$ is a left $A/A_\alpha$-module. Since $C_\alpha^* \iso A/A_\alpha$, 
Proposition~\ref{PropComodIsoMod} implies that $\Comod{}{C_\alpha} \iso \Mod{A/A_\alpha}{}$. 
Hence $N_\alpha$ is a right $C_\alpha$-comodule and so a right $C$-comodule. Finally, 
since $N = \bigcup N_\alpha$, $N$ is a union of $C$-comodules and therefore a $C$-comodule itself.

It is straightforward to check that these procedures are inverse to each other.
\end{proof}

\section{Regular coalgebras and their duals}
\label{SecRegCoalg}

In this section we will introduce the particular kind of coalgebra we work with.
Recall that the rational polynomial and Laurent polynomial rings $\Q[w]$
and $\Q[w,w^{-1}]$ are $\Z$-coalgebras with comultiplication determined by
making the indeterminate $w$ grouplike. We will be interested in $\Z$-subcoalgebras of these; 
we restrict to those which are free $R$-modules, where $R$ is either $\Z$ or $\Z_{(p)}$ for a prime $p$.

If $M$ is a free $R$-submodule of $\Q[w^r]$ (or $\Q[w^r, w^{-r}]$) for some $r \geq 1$, an $R$-basis for $M$ is
called \emph{regular} if it consists of polynomials $g_{n}$, for $n\in\N_0$ (or $n\in\Z$),
where the degree of the polynomial $g_{n}$ is $rn$. (See~\cite[II.1.3]{ivps}.)

\begin{defn}\label{DefnRegularCoalgebra}
Let $C$ be an $R$-subcoalgebra of $\Q[w^r]$ for some $r \geq 1$. Then we call $C$ a \emph{regular $R$-coalgebra} if
\begin{enumerate}
\item $C$ is a free $R$-module on a chosen regular basis $\{ c_n(w)\,|\,n \geq 0 \}$,
(where the degree of $c_n(w)$ is $rn$), and
\item $w^{rn} \in C$ for every $n \geq 0$.
\end{enumerate}
\end{defn}

\begin{exs}\label{ExsRegularCoalgebra}

\begin{enumerate}
\item For $R=\Z$ or $\Zpl$, $R[w^r]$ with basis $\{w^{rn}\,|\,n\geq 0\}$ is a regular coalgebra.
\item Taking $R=\Z$ and $r=1$, the free $R$-module $C$ with basis the binomial
coefficient polynomials
    $$
    \binom{w}{n}=\frac{w(w-1)\dots(w-n+1)}{n!}
    $$
is a regular coalgebra.
As is well-known, these polynomials form a basis for the ring of integer-valued
polynomials:
    $$
    C=\{f(w)\in\Q[w]\,|\,f(\Z)\subseteq \Z\}.
    $$
\end{enumerate}
\end{exs}
\smallskip

Further examples related to $K$-theory will appear in Section~\ref{SecKCoops}.
\medskip

Now let $C$ be a regular coalgebra with chosen regular basis $\{ c_n(w)\,|\,n \geq 0 \}$.
As will become clear, our methods are heavily dependent on the properties
of the specified basis elements; indeed this is why we have chosen to incorporate
the choice of basis into the definition. (While there are nice basis-free descriptions of
some of the objects we are interested in, we do not know how to prove our main results
without specifying bases.)

Thus
we introduce notation for the coefficients appearing in the basis elements, in the
comultiplication $\Delta$ applied to them and when
expressing powers of $w$ in terms of the basis elements.

\begin{defn}\label{DefnCoeffs}
For $k,n \geq 0$, define integers $\lambda_k^n$ and $D_n$ by writing $c_n(w)$ as
\[
c_n(w) = \frac{1}{D_n}\sum_{k=0}^n \lambda_k^n w^{kr},
\]
where 
$D_n > 0$ and the integers $\lambda^n_0, \dots,
\lambda^n_n, D$ have no non-trivial common factor.

\noindent Also define $\Lambda_n^k$ for $n,k \geq 0$ by
\[
w^{kr} = \sum_{n=0}^k \Lambda_n^k c_n(w).
\]
For $i,j,n \geq 0$, define $\Gamma_{i,j}^n$ by
\[
\Delta(c_n(w)) = \sum_{i,j=0}^n \Gamma_{i,j}^n c_i(w) \tp c_j(w).
\]
Where necessary, we will write $\Gamma(C)_{i,j}^n$, etc., to specify the coalgebra we are referring to.
\end{defn}

Next we consider the dual algebras of regular coalgebras. If $C$ is a 
regular $R$-coalgebra, define the finite rank subcoalgebras
\[
C_n = \{ f(w) \in C\,|\,\deg f(w) \leq rn \}
= R\{c_0(w),\dots,c_n(w)\},
\]
so that we have an increasing filtration $C_0 \ssne C_1 \ssne C_2 \ssne \dots$. 
Denote the dual $R$-algebra of $C$ by $A$. It has a \emph{topological $R$-basis} $\{ a_n\,|\,n \geq 0 \}$, 
where $a_n = c_n(w)^*$; \ie\ $A$ consists exactly of the infinite sums
\[
\sum_{n \geq 0} r_na_n
\]
for $r_n \in R$. For any $i,j,n \geq 0$, $\ab{a_ia_j,c_n(w)} = \Gamma^n_{i,j}$, so
the product of basis elements is given by
\[
a_ia_j = \sum_{n \geq 0} \Gamma^n_{i,j} a_n.
\]
The algebra $A$ has the topology given by the decreasing filtration by the ideals
\begin{align*}
A_n &= \Ann C_{n-1}= \{ a \in A\,|\,\ab{a,f(w)} = 0 \textup{ whenever $f(w) \in C_{n-1}$} \}\\
&= \Big\{ \sum_{k \geq n} r_ka_k\,|\,r_k \in R \Big\}.
\end{align*}
Thus $A$ is pro-finitely generated in the sense of Definition~\ref{Defnprofg}.

\begin{ex}
For $C=R[w^r]$ as in Example~\ref{ExsRegularCoalgebra}(1), the dual $A=C^*$ is the infinite
product $\prod_{i=0}^\infty R$, filtered by the ideals  $\prod_{i=n}^\infty R$.
\end{ex}

The coalgebra $D=C[w^{-r}]$ is the union of the finitely generated subcoalgebras
\[
D_n = D \cap \Q\{w^{-\floor{n/2}r},\dots,w^{\lceil n/2 \rceil r}\}
\]
for $n \geq 0$. Denote the dual algebra of $D$ by $B$; it also has a 
topological $R$-basis and may be filtered by the ideals $B_n = \Ann D_{n-1}$.
\medskip

For a regular coalgebra $C$, we can make the isomorphism of Proposition~\ref{PropDiscIsoComod}
between $C$-comodules and discrete $A$-modules very explicit.
The inverse 
functor $i^{-1} \from \Disc{A}{} \to \Comod{}{C}$ 
to $i:\Comod{}{C}\to \Disc{A}{}$ is given
explicitly on objects as follows. (On morphisms, 
of course, it is the identity.) If $N$ is a discrete $A$-module, then $i^{-1}(N)$ 
has the $C$-coaction $\rho_{i^{-1}(N)} \from i^{-1}(N) \to i^{-1}(N) \tp C$, where
\[
\rho_{i^{-1}(N)}(y) = \sum_{n \geq 0} a_ny \tp c_n(w).
\]
This sum has only finitely many terms because $N$ is discrete.
\medskip

An element $a$ of $A$ is determined by the sequence $(\ab{a,w^{rn}})_{n \geq 0}$ of 
elements of $R$. We use this fact to characterize the units in $A$.

\begin{lemma}\label{LemmaDetectingUnits}
Let $C$ be a regular $R$-coalgebra and $A = C^*$. Set $D = C[w^{-r}]$ and $B = D^*$. Then
\begin{enumerate}
\item an element $a$ of $A$ is a unit if and only if $\ab{a,w^{rn}} \in R^\times$ for each $n \geq 0$.
\item an element $b$ of $B$ is a unit if and only if $\ab{b,w^{rn}} \in R^\times$ for each $n \in \Z$.
\end{enumerate}
\end{lemma}

\begin{proof}
(See \cite[Theorem 3.8]{ccw2}.) We do part (1); part (2) is similar. 
If $a \in A^\times$ then clearly each $\ab{a,w^{rn}} \in R^\times$. Conversely, 
take $a \in A$ with $\ab{a,w^{rn}} \in R^\times$ for each $n \geq 0$ and write 
$a = \sum_{k \geq 0} r_k a_k$. Suppose inductively that 
we have $s_0, \dots, s_{i-1} \in R$ such that
\[
a(s_0 + s_1 a_1 + \dots + s_{i-1} a_{i-1}) \in 1 + A_i.
\]

For any $s_i$, the element $a(s_0 + s_1 a_1 + \dots + s_i a_i)$ is in
$1 + A_i$ and has $a_i$-coefficient
\[
s_i \Big(\sum_{k=0}^i r_k \Gamma_{k,i}^i\Big) + \textup{(terms not involving
$s_i$)},
\]
so we may choose $s_i$ such that $a(s_0 + s_1 a_1 + \dots + s_i a_i)$ is in $1 + A_{i+1}$ 
(\ie\ such that the above displayed expression is zero) provided $\sum_{k=0}^i r_k \Gamma_{k,i}^i$ is a unit.

Now,
\begin{align*}
\Delta(c_n(w)) &= \frac{1}{D_n} \sum_{k=0}^n \lambda_k^n w^{kr} \tp w^{kr}\\
&= \frac{1}{D_n} \sum_{k=0}^n \lambda_k^n \sum_{r,s=0}^k \Lambda_r^k \Lambda_s^k c_r(w) \tp c_s(w)\\
&= \sum_{r,s=0}^n \left( \sum_{k=0}^n \frac{\Lambda_r^k \Lambda_s^k \lambda_k^n}{D_n} \right) c_r(w) \tp c_s(w),
\end{align*}
so, in particular, $\Gamma_{n,i}^i = \Lambda_n^i \Lambda_i^i \lambda_i^i / D_i = \Lambda_n^i$ 
for any $i$ and $n$ (since all other terms are zero and $\Lambda_i^i \lambda^i_i = D_i$). Hence
\begin{align*}
\ab{a,w^{ir}} &= \sum_{k \geq 0} r_k \ab{a_k,w^{ir}}
= \sum_{k=0}^i r_k \sum_{j=0}^i \Lambda_j^i \ab{a_k,c_j(w)}\\
&= \sum_{k=0}^i r_k \Lambda_k^i
= \sum_{k=0}^i r_k \Gamma_{k,i}^i,
\end{align*}
as required.
\end{proof}

\section{$K$-theory cooperations and operations}
\label{SecKCoops}

In this section we will describe several examples from topology fitting into the framework of Section~\ref{SecRegCoalg}.
For those unfamiliar with the topological context of our examples we give a brief description.

In algebraic topology we work with generalized cohomology theories. Such a theory
$E^*$ is a well-behaved functor from the homotopy category of spaces (or spectra) to the
category of graded abelian groups. An important example is provided by periodic complex
$K$-theory $K^*$. 

The \emph{operations} of the theory $E^*$ are the natural transformations of the functor. We will
focus on those operations which are \emph{stable}, meaning that they are compatible with the
suspension functor in a standard way. Such stable operations, by the Yoneda Lemma, are given by
$E^*(E)=[E,E]$, the homotopy classes of maps from $E$ to itself, where $E$ is a spectrum
representing the cohomology theory $E^*$. In our examples, we will focus on the \emph{degree zero}
stable 
operations, given by $E^0(E)$. This has an algebra structure in which multiplication is given
by composition of operations.

One can also consider $E_*(E)=\pi_*(E\smash E)$, the so-called \emph{cooperations} of the theory
$E$, and $E_0(E)=\pi_0(E\smash E)$, the degree zero cooperations. In favourable cases, such as those
we study, $E^0(E)$ is the linear dual of $E_0(E)$.

We will take $E$ to be one of various cohomology theories related to 
periodic complex $K$-theory $K^*$.
The coalgebras we study are the degree zero cooperations $E_0(E)$
for such $E$. The comultiplication is dual to composition of operations.

Our notation for the theories we use is as follows. 
We let $K_{(p)}$ denote the periodic complex $K$-theory
spectrum with $p$-local coefficients; $k_{(p)}$ is the corresponding connective spectrum.
For an odd prime $p$, it was proved by Adams in~\cite[Lecture 4]{adams69} that 
these spectra split as wedges of suspensions of simpler spectra:
    $$
    K_{(p)}\simeq \bigvee_{i=0}^{p-2} \Sigma^{2i}G, \quad
    k_{(p)}\simeq \bigvee_{i=0}^{p-2} \Sigma^{2i}g.  
    $$
The spectra $G$ and $g$ appearing in these splittings are known as the
Adams summands of p-local periodic and connective $K$-theory respectively. We write
$KO$ and $ko$ for periodic and connective real $K$-theory respectively,
and $KO_{(2)}$, $ko_{(2)}$ for their $2$-localizations.
Let $R=\Z$ or $\Z_{(p)}$ as appropriate.

Writing $F$ for any of the periodic theories and $f$ for the corresponding connective
theory, we have that the degree zero stable operation algebra $F^0(F)$ 
is the $R$-linear dual of the degree zero cooperations $F_0(F)$. Similarly,
$f^0(f)$ is the $R$-linear dual of $F_0(f)$. 

Rational Laurent polynomials arise as the degree zero cooperations for rational
periodic $K$-theory:
\[
K_0(K)\otimes\Q\cong \Q[w, w^{-1}].
\]
Here $w$ can be thought of as $uv^{-1}$ where $K_*(K)\otimes\Q\cong \Q[u,v,u^{-1},v^{-1}]$
with $u,v\in K_2(K)$.

To simplify notation somewhat we will write, for example, $K_0(K)_{(p)}$
for ${K_{(p)}}_0(K_{(p)})$ and $K^0(K)_{(p)}$ for
${K_{(p)}}^0(K_{(p)})$. In the case of operations, this is an abuse
of notation, in the sense that ${K_{(p)}}^0(K_{(p)})$ is isomorphic to
a completed tensor product $K^0(K)\widehat{\otimes}\Zpl$,
not to $K^0(K){\otimes}\Zpl$.

All of our examples of cooperations are torsion-free and can naturally be viewed as
subalgebras of $\Q[w, w^{-1}]$. In the cooperations world
the relation between $F_0(F)$ and $F_0(f)$ is given by $F_0(F)=F_0(f)[w^{-r}]$ for suitable
$r$ (namely $r=1$ for $F=K_{(p)}$, $r=2$ for $F=KO_{(2)}$ and $r=p-1$ for $F=G$).

We will see that our examples are regular coalgebras in the
sense of the previous section.
In order to describe their bases we need some notation for polynomials.

\begin{defn}\label{Defnthetapolys}
If $n \geq 0$ and $(z_i)_{i\geq 1}$ is a sequence of elements of $\Q$, define the polynomial
\[
\theta_n(X;(z_i)) = \prod_{i=1}^n (X - z_i).
\]
Also for any $z\in\Q$, write $\theta_n(X;z)$ for
\[
\theta_n(X;(z^{i-1})) = (X-1)(X-z)\dots(X-z^{n-1}).
\]
\end{defn}

We now summarize some known results giving regular bases for cooperation coalgebras.
In all of our examples, bases can be expressed in terms of the above $\theta$-polynomials
for suitable choices of the sequence $(z_i)$.
See also~\cite{strongw} where these examples are discussed 
 in terms of integer-valued polynomials.

\begin{thm}

\begin{enumerate}\label{basisconnectivecoops}
\item
\cite[Prop.~3]{ccw}.
Let $p$ be an odd prime and let $q$ be primitive mod $p^2$. 
Then $K_0(k)_{(p)}$ is a regular $\Zp$-coalgebra with $\Zp$-basis 
\[
f_n(w) = \frac{\theta_n(w;q)}{\theta_n(q^n;q)}, \qquad\text{for }n \geq 0.
\]
\item
\cite[4.2]{ccw2}.
Let $p$ be an odd prime, let $q$ be primitive mod $p^2$ and let $\hat{q} = q^{p-1}$. Then
$G_0(g)$ is a regular $\Zp$-coalgebra with $\Zp$-basis 
\[
\hat{f}_n(w) = \frac{\theta_n(w^{p-1};\hat{q})}{\theta_n(\hat{q}^n;\hat{q})}, \qquad\text{for }n \geq 0.
\]
\item 
\cite[9.2]{ccw2}.
$KO_0(ko)_{(2)}$ is a regular $\Z_{(2)}$-coalgebra with $\Z_{(2)}$-basis
\[
h_n(w) = \frac{\theta_n(w^2;9)}{\theta_n(9^n;9)}, \qquad\text{for }n \geq 0.
\]
\item 
\label{PropBForK0k2}
\cite[Prop.~20]{ccw}
$K_0(k)_{(2)}$ is a regular $\Z_{(2)}$-coalgebra with $\Z_{(2)}$-basis 
\[
f_{2m}^{(2)}(w) = h_m(w)
\]
and
\[
f_{2m+1}^{(2)}(w) = \frac{3^m-w}{2.3^m}h_m(w)
\]
for $m \geq 0$.\qed
\end{enumerate}
\end{thm}

The dual topological bases 
can all be written as polynomials in Adams operations, which arise in this
context as evaluation maps.

\begin{defn}
For $\beta\in\Q$ let the \emph{Adams operation}
$\Psi^\beta:\Q[w,w^{-1}]\to\Q$ be the evaluation map,
$\Psi^\beta(f(w))=f(\beta)$. 
\end{defn}

We also write $\Psi^\beta$ for the restriction of this map to an
$R$-subcoalgebra $C$ of $\Q[w, w^{-1}]$. For suitable choices of $\beta$ this
evaluation map will take values in $R\subset\Q$ and so can be viewed as an element
of the dual algebra to $C$.

\begin{thm}\label{ThmTBForOps}
\begin{enumerate}
\item
\cite[2.2]{ccw2}.
If $p$ is odd and $q$ is primitive mod $p^2$, 
\[
\{\theta_n(\Psi^q;q)\,|\, n\geq 0\}
\] 
is the dual topological $\Zp$-basis for $k^0(k)_{(p)}$ 
to the basis $f_n(w)$ for $K_0(k)_{(p)}$.
\item 
\cite{Lellmann}. 
If $p$ is odd, $q$ is primitive mod $p^2$ and $\hat{q} = q^{p-1}$, 
\[
\{\theta_n(\Psi^q;\hat{q})\,|\, n\geq 0\}
\]
is the dual topological $\Zp$-basis for 
$g^0(g)$ to the basis $\hat{f}_n(w)$ for $G_0(g)$.
\item 
\cite[9.3]{ccw2}. 
The dual topological $\Z_{(2)}$-basis for $ko^0(ko)_{(2)}$ to the basis $h_n(w)$ for $KO_0(ko)_{(2)}$
is  $\{\theta_n(\Psi^3;9)\,|\,n\geq 0\}$.
\qed
\end{enumerate}
\end{thm}

Using the method of Adams and Clarke~\cite{ac} one obtains the following bases for the periodic
versions.

\begin{thm}
\label{periodiccoopsbases}
\begin{enumerate}
\item
\cite[Cor.~6]{ccw}. 
If $p$ is an odd prime, then the Laurent polynomials
$w^{-\floor{n/2}}f_n(w)$,
for $n \geq 0$, form a $\Z_{(p)}$-basis for $K_0(K)_{(p)}$.
\item If $p$ is an odd prime,
$\{w^{-\floor{n/2}(p-1)}\hat{f}_n(w)\,|\, n\geq 0\}$ is a $\Z_{(p)}$-basis for $G_0(G)$.
\item $\{ w^{-n}h_n(w)\,|\,n\geq0\}$ is a $\Z_{(2)}$-basis for $KO_0(KO)_{(2)}$.
\item $\{ F^{(2)}_n(w) = w^{-\floor{n/2}}f^{(2)}_n(w)\,|\,n\geq 0\}$ is a $\Z_{(2)}$-basis for $K_0(K)_{(2)}$.
\qed
\end{enumerate}
\end{thm}

The next step is to give topological bases for the dual algebras of these periodic objects. Define the polynomials
\[
\Theta_n(X;q) = \theta_n\big(X;\big(q^{(-1)^i\floor{i/2}}\big)\big).
\]

\begin{thm}\label{ThmTBForPeriodicOps}
\begin{enumerate}
\item
\label{ThmTBForK0Kp}
\cite[6.2]{ccw2}.
The set $\{ \Theta_n(\Psi^q;q)\,|\, n \geq 0 \}$ is a topological $\Z_{(p)}$-basis for $K^0(K)_{(p)}$.
\item
\label{ThmTBForgAndG}
The set $\{ \Theta_n(\Psi^q;\hat{q})\,|\,n \geq 0 \}$ is a topological $\Z_{(p)}$-basis for $G^0(G)$.
\item 
\cite[9.3]{ccw2}. The set $\{ \Theta_n(\Psi^3;9)\,|\,n \geq 0 \}$ is a topological $\Z_{(2)}$-basis for $KO^0(KO)_{(2)}$.\qed
\end{enumerate}
\end{thm}

It is a bit more complicated to describe topological bases for $k^0(k)_{(2)}$ and $K^0(K)_{(2)}$. 
The former is done in~\cite[8.2]{ccw2} and the latter can be done using Theorem 5.2 of~\cite{strongw}, 
but this does not give the
answers in any particularly nice form. Instead our strategy for these cases will be to
work directly with the cooperations.

\section{Locally finitely generated modules}
\label{SecLFG}

In this section we study modules which are locally finitely generated over the ground ring $R$. 
As we explain, it is easy to see that discrete modules
are locally finitely generated over $R$. We will give conditions under which the converse is true. 

Consider an $R$-algebra $A$ with a filtration $A_\alpha$ satisfying the conditions of Definition~\ref{Defnprofg}
and a discrete $A$-module $M$. 
We will consider only $R = \Z$ and $R = \Zp$ for $p$ prime. If $x \in M$, by Lemma~\ref{LemmaCharDisc},
there is $\alpha$ with $A_\alpha x=0$. 
Then $Ax = (A/A_\alpha)x$, so $Ax$ is finitely generated as an $R$-module, as $A/A_\alpha$ is.

\begin{defn}
A module $M$ over an $R$-algebra $A$ is \emph{locally finitely generated over $R$} if 
$Ax$ is a finitely generated $R$-module for every $x \in M$. The category 
$\LFG{A}{}$ is the full subcategory of $\Mod{A}{}$ whose objects are the 
$A$-modules which are locally finitely generated over $R$.
\end{defn}

Of course, if $A$ is itself finitely generated as an $R$-module then every $A$-module
is locally finitely generated over $R$. The definition is of interest in the case where
$A$ is not itself finitely generated over $R$. In this case, $A$ regarded as a module over
itself is clearly not locally finitely generated.
\medskip

The above discussion shows that $\Disc{A}{} \ss \LFG{A}{}$. In this section we give our main technical result, 
Theorem~\ref{ThmLFGIsDiscrete}, providing 
conditions under which these subcategories of $\Mod{A}{}$ are equal. 
At the end of the section, Theorem~\ref{ThmLFGIsDiscrete(C)} gives a 
version of the conditions formulated more directly in terms of the structure of a regular coalgebra
 to which $A$ is dual.

In the next section, we show that our method applies to the algebras of stable operations 
of many variants of topological $K$-theory. 

\begin{thm}\label{ThmLFGIsDiscrete}
Let $p$ be a prime and $A$ be a topological $\Z_{(p)}$-algebra with a topological 
basis $\{ a_n\,|\,n \geq 0 \}$. If for every $l > 0$ there is an infinite set $N_l \ss \N_0$ such that
\begin{enumerate}
\item $1-a_{n-m}$ is a unit in $A$ for every $m,n \in N_l$ with $m<n$, and
\item $a_m a_n \equiv a_{m+n} \text{ mod } p^l$ whenever $m \in N_l$ and $n \in \N_0$,
\end{enumerate}
then $\Disc{A}{} = \LFG{A}{}$.
\end{thm}

Of course, if $A=\Z[\![t]\!]$ with basis $\{t^n\,|\, n\geq 0\}$, then the conditions of the
theorem are satisfied with $N_l=\N$ for all $l$.
\medskip

We need some preliminary results before we prove Theorem \ref{ThmLFGIsDiscrete}. 
First, we can make a simplification. Suppose that $M$ is an $A$-module which is locally finitely generated over $R$
and let $x \in M$. Then $x$ is contained in an $R$-finitely generated $A$-submodule of $M$, 
namely $Ax$. If every $Ax$ is discrete, then $x \in Ax \ss \Dsc M$ for every $x \in M$, so $\Dsc M = M$
and $M$ is discrete. Thus it is enough to show that every $R$-finitely generated $A$-module is discrete.

Let $N$ be an $R$-finitely generated $A$-module. We consider separately the cases of $N$ a free $R$-module of 
finite rank and $N$ a finite torsion $R$-module and then prove a limited extension 
theorem to combine the two. Take first an $N$ which is free over $R$. Discreteness 
will follow from the fact that this is a `slender' abelian group.

\begin{defn}
\cite[\S94]{FuchsII}. Let
\[
P = \prod_{n \geq 0} \Z e_n
\]
for some generators $e_n$. An abelian group $G$ is \emph{slender} if every linear map 
$\eta \from P \to G$ has $\eta(e_n)=0$ for all but finitely many $n$.
\end{defn}

In the literature, \eg\ in \cite{FuchsII}, this property is only 
considered for $G$ torsion-free, but we can define it for every $G$. It 
is relevant to us because $A$ contains a subgroup isomorphic to
$P$, under the map $e_n \mapsto a_n$, so any linear map out of $A$ induces 
a linear map out of $P$.
In particular this applies to the map $a \mapsto ax$ from $A$ to $M$,
where $M$ is an $A$-module and $x \in M$.

\begin{lemma}
\cite[94.2]{FuchsII}. If $R = \Z$ or $R = \Zp$, then a free $R$-module of finite rank is
slender.\qed
\end{lemma}

\begin{lemma}\label{LemmaCIsoA*}
If $R = \Z$ or $R = \Zp$ and $C$ is a regular $R$-coalgebra, the $R$-linear map $C \to A^*$ determined by
\[
c_n(w) \mapsto a_n^*
\]
is an isomorphism of abelian groups.
\end{lemma}

\begin{proof}
Certainly this map is injective. If $\alpha \from A \to \Zp$ is a linear map, 
then by slenderness of $R$, $\alpha(a_n) = 0$ for all but finitely many $n$. 
Hence $\alpha$ is a finite linear combination of the $a_n^*$ and so lies in the image of $C$.
\end{proof}

Generally, we would expect $A^*$ to be larger than $A$, just as $A = C^*$ is larger than $C$. 
This odd property of slender groups solves the first part of our problem.

\begin{prop}\label{PropFreeIsDiscrete}
An $A$-module $N$ which is free of finite rank as an $R$-module is discrete.
\end{prop}

\begin{proof}
Take $x \in N$ and let $x_1,\dots,x_k$ be an $R$-basis for $N$. Then if $a \in A$,
\[
ax = \sum \gamma_i(a,x) x_i
\]
for some $\gamma_i(a,x) \in R$. For $i = 1,\dots,k$, define $R$-linear maps
\[
g_i(x) \from A \to R
\]
by $g_i(x)(a) = \gamma_i(a,x)$ and let $g'_i(x)$ be the element of $C$ 
corresponding to $g_i(x)$ under the isomorphism of Lemma \ref{LemmaCIsoA*}. 
The $g'_i(x)$ give us a map $\rho \from N \to N \tp C$,
\[
\rho(x) = \sum x_i \tp g'_i(x),
\]
which is $R$-linear because $a(x+y) = ax+ay$. Such a map makes $N$ a $C$-comodule 
if and only if the corresponding map (using the action \eqref{EqComodIsMod}) 
$A \tp N \to N$ is an $A$-action map. But this map is the given action of $A$ on $N$, 
so $N$ is indeed a $C$-comodule. By Lemma~\ref{LemmaComodsAreDiscrete}, $N$ is discrete.
\end{proof}

We cannot use the concept of slenderness to attack the case of finite $A$-modules, since no
 non-zero finite $\Zp$-module is slender. (One can see this by showing that
for $M$ a non-zero finite $\Zp$-module, $\Hom_{\Z}(P,M)$ is an uncountable set.)
\smallskip

The remainder of our approach is motivated by the case of $\Z[\![t]\!]$. 
If $N$ is a $\Z$-finitely generated $\Z[\![t]\!]$-module, we can use the 
classification of finitely generated abelian groups to write $N$ 
as an extension of $\Z[\![t]\!]$-modules
\[
0 \to T \to N \to F \to 0,
\]
where $T$ is the torsion part of $N$ and $F$ is free over $\Z$. 
By Proposition~\ref{PropFreeIsDiscrete}, 
$F$ is discrete. 
If $x \in T$ then since $T$ is finite there must be $m < n$ such that $t^mx = t^nx$. 
Then $(t^m - t^n)x = 0$, so $t^m(1 - t^{n-m})x = 0$. Since $1 - t^{n-m}$ is a unit in $\Z[\![t]\!]$, 
it follows that $t^mx=0$. So $(t^m)x=0$ and therefore $T$ is discrete.
As discussed earlier, over $\Z[\![t]\!]$, discrete modules are closed under extensions, so
$N$ is discrete. The conditions of Theorem~\ref{ThmLFGIsDiscrete} allow us to generalize this argument.

\begin{lemma}\label{LemmaFiniteIsTame}
If $A$ satisfies conditions (1) and (2) of Theorem~\ref{ThmLFGIsDiscrete}, 
$M$ is a finite $A$-module and $p^sM = 0$, then there is $m \in N_s$ with $a_mM = 0$.
\end{lemma}

\begin{proof}
Let $x \in M$. Since $N_s$ is infinite and $M$ is finite, 
there must be $m,n \in N_s$ with $m < n$ and $a_mx=a_nx$, equivalently,
\[
(a_m-a_n)x=0.
\]
We may factorize this as $a_m(1-a_{n-m})x=p^s\theta x= 0$ for 
some $\theta \in A$ using (2); then, by (1), $a_mx=0$.

Now if we have such an $m_x$ for each $x \in M$, let $m$ be the largest. 
For any $x$, $a_mx = a_{m-m_x}a_{m_x}x = 0$, because $m_x \in N_s$.
\end{proof}

\begin{lemma}\label{LemmaExtensionsLemma}
If $A$ satisfies condition (2) of Theorem~\ref{ThmLFGIsDiscrete}, 
$m \in N_l$ and $n \geq 0$, then for any $a \in A_{m+n}$ there is 
$b \in a_mA_n$ such that $p^l \text{ divides } a-b$.
\end{lemma}

\begin{proof}
We may write
\[
a = \sum_{i \geq n} \lambda_{m+i} a_{m+i}
\]
for some $\lambda_{m+i} \in \Zp$. By (2), for each $i$ there is $\theta_i \in A$ with
\[
a_{m+i} = a_m a_i + p^l \theta_i.
\]
Note that, since $A_i$ is an ideal, $\theta_i \in A_i$. Hence the 
infinite sum $\sum_{i \geq n} \lambda_{m+i} \theta_i$ converges and therefore represents an element of $A$.

Summing over $i$, we get
\begin{align*}
a &= \sum_{i \geq n} \lambda_{m+i} (a_m a_i + p^l \theta_i)\\
&= a_m \sum_{i \geq n} \lambda_{m+i} a_i + p^l \sum_{i \geq n}
\lambda_{m+i} \theta_i,
\end{align*}
as required.
\end{proof}

\begin{proof}[Proof of Theorem~\ref{ThmLFGIsDiscrete}]
We have reduced this to showing that a $\Zp$-finitely generated
$A$-module $N$ is discrete. Consider the extension
\[
0 \to T \to N \to F \to 0
\]
of $A$-modules, where $T$ is the $p$-torsion $A$-submodule of $N$
and $F = N/T$. By the classification of finitely generated
$\Zp$-modules~\cite[IV.6]{Hungerford}, $F$ is a free $\Zp$-module of finite rank;
hence it is discrete by Proposition~\ref{PropFreeIsDiscrete}. Take
$n$ with $A_nF=0$. (Since $F$ is $\Zp$-finitely generated and hence $A$-finitely generated, there is
such an $n$.)

Now take $x \in N$. Since $A_n(x+T)=0$ in $F$, we must have $A_nx \ss T$.
By Lemma \ref{LemmaFiniteIsTame}, there is $m \in N_s$, where
$p^sT=0$, such that $a_mT = 0$; hence $a_mA_nx = 0$.

We want to show that $A_{m+n}x = 0$, so take $a \in A_{m+n}$. By Lemma~\ref{LemmaExtensionsLemma}, 
there are $b \in a_mA_n$ and $\theta \in A$ such that $a - b = p^s\theta$. Then
$ax = bx + p^s \theta x= p^s \theta x$,
as $a_mA_n x = 0$. We also know that $ax \in T$, because $A_{m+n}x \ss A_nx \ss T$, so $p^s ax = 0$. 
This means $p^{2s}\theta x = 0$, so $\theta x \in T$. But then $p^s \theta x = 0$, \ie\ $ax=0$.
\end{proof}

We end this section with a result corresponding to Theorem~\ref{ThmLFGIsDiscrete}
under the assumption that we are given a regular coalgebra $C$ or $D = C[w^{-r}]$. 
This form of the result is less intuitive, but we will need it to deal with those applications in 
the next section where the dual topological basis is impractical to work with. 
We refer to Definition~\ref{DefnCoeffs} for the coefficients $\Lambda$ and $\Gamma$
appearing in the statement.

\begin{thm}\label{ThmLFGIsDiscrete(C)}
Let $p$ be a prime. Let $C$ be a regular $\Z_{(p)}$-coalgebra with specified basis
$\{c_n(w)\,|\, n\geq 0\}$ and let $A=C^*$.
If for every $l > 0$ there is an infinite set $N_l \ss \N_0$ such that
\begin{enumerate}
\item $p \text{ divides } \Lambda_{n-m}^j$ whenever $m,n \in N_l$ with $m<n$ and $j \in \N_0$, and
\item if $m \in N_l$ and $n \in \N_0$, then 
$\Gamma_{m,n}^{m+n} = 1$ and $p^l \text{ divides } \Gamma_{m,n}^i$ whenever $i \neq m+n$,
\end{enumerate}
then $\Disc{A}{} = \LFG{A}{}$.

If $D = C[w^{-r}]$ has a $\Zp$-basis $F_n(w)$ which satisfies analogous conditions 
to those above and $B = D^*$, then $\Disc{B}{} = \LFG{B}{}$.
\end{thm}

\begin{proof}
We prove the first statement; the second is similar. It is enough to check that 
these conditions imply those of Theorem \ref{ThmLFGIsDiscrete}. For any $n > m \geq 0$ and $j \geq 0$,
\[
\ab{1-a_{n-m},w^{jr}} = 1 - \Big\langle a_{n-m},\sum_{i=0}^j \Lambda_i^j f_i(w) \Big\rangle
= 1 - \Lambda_{n-m}^j,
\]
which, by (1), is in $1+p\Z_{(p)} \ss \Z_{(p)}^\times$ provided $m,n \in N_l$. 
By Lemma \ref{LemmaDetectingUnits}(1), therefore, $1-a_{n-m}$ is a unit, so condition (1) 
of Theorem~\ref{ThmLFGIsDiscrete} is satisfied.

Since
\[
a_m a_n = \sum_{i \geq 0} \Gamma_{m,n}^i a_i,
\]
condition (2) above immediately implies condition (2) of Theorem \ref{ThmLFGIsDiscrete}.
\end{proof}

\section{Applications of the theorems}\label{SecApplications}

In this section we apply our results on locally finitely generated modules 
to the algebras of operations described in Section~\ref{SecKCoops}.
We will use Theorems~\ref{ThmLFGIsDiscrete} and~\ref{ThmLFGIsDiscrete(C)} to prove the following result.

\begin{thm}\label{ThmLFGForKs}
Let $p$ be a prime. If $E$ is $K_{(p)}$, $k_{(p)}$, $G$, $g$, $KO_{(2)}$ or $ko_{(2)}$, then
\[
\Disc{E^0(E)}{} = \LFG{E^0(E)}{}.
\]
\end{thm}

Note that the cases of $K_{(2)}$ and $k_{(2)}$ are included in this statement; 
we will, however, have to prove them separately from $K_{(p)}$ and $k_{(p)}$ for 
$p$ odd. In each case we need to find suitable infinite subsets $N_l$ of $\N_0$, 
one for each $l > 0$. We list them in the following table. In general we must discard small 
values of $n$ in order for condition (1) of Theorem~\ref{ThmLFGIsDiscrete} to be satisfied
and impose a divisibility condition for condition (2) to be satisfied. For $K_{(2)}$ and $k_{(2)}$ 
we will, as noted before, use Theorem~\ref{ThmLFGIsDiscrete(C)}.
\newpage

\begin{tabel}\label{table}
\end{tabel}
\begin{center}
\renewcommand{\arraystretch}{0.6}
\begin{tabular}{c|c}
Spectrum            & $N_l$\\
\hline
\rule{0pt}{14pt}
$K_{(p)}$ ($p$ odd) & $\{ n \in \N_0\,|\,n \geq p-1 \textup{ and } 2p^{l-1}(p-1) \text{ divides } n \}$\\
\\
$k_{(p)}$ ($p$ odd) & $\{ n \in \N_0\,|\,n \geq p-1 \textup{ and } p^{l-1}(p-1) \text{ divides } n \}$\\
\\
$G$                 & $\{ n \in \N_0\,|\,n \geq 1 \textup{ and } 2p^{l-1} \text{ divides } n \}$\\
\\
$g$                 & $\{ n \in \N_0\,|\,n \geq 1 \textup{ and } p^{l-1} \text{ divides } n \}$\\
\\
$KO_{(2)}$          & $N_1 = N_2 = \{ n \in \N_0\,|\,n \geq 1 \textup{ and } 2 \text{ divides } n \}$,\\
                    & $\{ n \in \N_0\,|\,n \geq 1 \textup{ and } 2^{l-2} \text{ divides } n \}$ for $l\geq 3$\\
\\
$ko_{(2)}$          & $N_1 = N_2 = \{ n \in \N_0\,|\,n \geq 1 \}$,\\
                    & $\{ n \in \N_0\,|\,n \geq 1 \textup{ and } 2^{l-3} \text{ divides } n \}$ for $l\geq 3$\\
\\
$K_{(2)}$        & $N_1 = N_2 = \{ n \in \N_0\,|\,n \geq 1 \textup{ and } 2 \text{ divides } n \}$,\\   
                    & $\{ n \in \N_0\,|\,n \geq 1 \textup{ and } 2^{l-2} \text{ divides } n \}$ for $l\geq 3$\\
\\
$k_{(2)}$           & $\{ n \in \N_0\,|\,n \geq 1 \textup{ and } 2^{l-3} \text{ divides } n \}$
\end{tabular}
\end{center}
\medskip

Before we go through the cases, we prove a useful lemma.
Let $(z_i)_{i\geq 1}$ be a sequence of rationals.
Recall the polynomials $\theta_n(X;(z_i))$ from Definition~\ref{Defnthetapolys}.
Note that, for any $0 \leq i \leq m$, the quotient
\[
\frac{\theta_m(X;(z_i))}{\theta_{m-i}(X;(z_i))} = \prod_{k=m-i+1}^m (X-z_k)
\]
is a polynomial in $X$.

\begin{lemma}\label{LemmaThetasAreLikePolys}
Let $(z_i)_{i\geq 1}$ be a sequence of rationals and, for brevity, 
write $\theta_n(X)$ for $\theta_n(X;(z_i))$. Then, for any $m,n \geq 0$,
\[
\theta_{m+n}(X) = \theta_m(X)\theta_n(X) + \sum_{i=0}^{n-1} (z_{n-i}-z_{m+n-i}) 
\frac{\theta_n(X)}{\theta_{n-i}(X)} \theta_{m+n-i-1}(X).
\]
\end{lemma}

\begin{proof}
For $0 \leq i \leq n-1$,
\begin{align*}
\theta_{m+n-i}(X) &= (X-z_{m+n-i})\theta_{m+n-i-1}(X)\\
&= (X-z_{n-i})\theta_{m+n-i-1}(X) + (z_{n-i}-z_{m+n-i})\theta_{m+n-i-1}(X).
\end{align*}
Repeatedly apply this formula.
\end{proof}

Next we give conditions which imply those of Theorem~\ref{ThmLFGIsDiscrete} specifically for
the case where $A$ is the dual of a regular coalgebra and has a topological basis of
$\theta$-polynomials. We write $\nu_p$ for $p$-adic valuation.

\begin{thm}\label{ThmLFGIsDiscrete-specific}
Let $A$ be the dual of a regular $\Zp$-coalgebra $C$ and suppose that $A$ has a topological $\Zp$-basis of
the form $\{a_n=\theta_n(\Psi^\beta, (z_i))\,|\,n\geq 0\}$ for some $\beta, z_i\in\Zp$.
If for every $l > 0$ there is an infinite set $N_l \ss \N_0$ such that
\begin{enumerate}
\item $\theta_{n-m}(\beta^{jr}, (z_i))\in p\Zp$  for every $m,n \in N_l$ with $m<n$ and $j\in\N_0$, and
\item $\nu_p(z_{n-i}-z_{m+n-i})\geq l$ for $0\leq i\leq n-1$ whenever $m \in N_l$ and $n \in \N_0$,
\end{enumerate}
then $\Disc{A}{} = \LFG{A}{}$.
\end{thm}

\begin{proof}
Since $\langle 1-a_{n-m}, w^{jr}\rangle = 1-\theta_{n-m}(\beta^{jr}, (z_i))$, using Lemma~\ref{LemmaDetectingUnits},
we see that condition (1) here implies condition (1) of Theorem~\ref{ThmLFGIsDiscrete}. It follows directly
from Lemma~\ref{LemmaThetasAreLikePolys} that condition (2) here implies condition (2) of Theorem~\ref{ThmLFGIsDiscrete}.
\end{proof}

We now give the details of how to apply Theorem~\ref{ThmLFGIsDiscrete-specific}
to most of the connective theories.

\begin{prop}\label{PropLFGconnective}
Let $E$ be $k_{(p)}$ for $p$ an odd prime, $g$ or $ko_{(2)}$. Consider the topological
bases for $A=E^0(E)$ given in Theorem~\ref{ThmTBForOps}
and let the sets $N_l$ for $l>0$ be as specified in Table~\ref{table}.
Then conditions (1) and (2) of Theorem~\ref{ThmLFGIsDiscrete-specific} are satisfied.
\end{prop}

\begin{proof}
The three theories fit into the framework of Theorem~\ref{ThmLFGIsDiscrete-specific} with
data as follows.
\vspace{-0.3cm}
\begin{center}
\begin{tabular}{c|ccc}
$E$&$r$&$\beta$&$z_i$\\
\hline
$k_{(p)}$&$1$&$q$&$q^{i-1}$\\
$g$&$p-1$&$q$&$\hat{q}^{i-1}$\\
$ko_{(2)}$&$2$&$3$&$9^{i-1}$
\end{tabular}
\end{center}

First consider $E=k_{(p)}$.
Let $j\geq 0$. Taking $n\geq p-1$, there is some $0\leq i\leq n-1$ such that $p-1$ divides
$j-i$ and so $p$ divides $q^j-q^i$. Thus $\theta_{n}(\beta^{jr}, (z_i))=\prod_{i=0}^{n-1}(q^j-q^i)\in p\Zp$
and this shows condition (1) is satisfied.

Taking $l >0$, $n \geq 0$ and $m \in N_l$, we have
$q^{n-i-1}-q^{m+n-i-1} = q^{n-i-1}(1-q^m)$ and since $p^{l-1}(p-1)$ divides $m$, $p^l$ divides $1-q^m$. 
So condition (2) is satisfied.
\smallskip

Next we consider $E=g$.
For each $n \geq 1$ and $j \geq 0$,
$\theta_n(\hat{q}^j;\hat{q})$ lies in $p\Zp$, and condition (1) is satisfied.
Since $\nu_p(\hat{q}^{n-i-1}-\hat{q}^{m+n-i-1}) = \nu_p(1-\hat{q}^m)
= 1+\nu_p(m)$, condition (2) is also satisfied.
\smallskip

Finally, let $E=ko_{(2)}$.
Recall that
\begin{equation}\label{Eq2adic}
\nu_2(3^i-1) = \begin{cases}1 & \textup{if $i$ is odd;}\\2+\nu_2(i) & \textup{if $i$ is even.}\end{cases}
\end{equation}
Thus $\theta_n(9^l;9)\in p\Zp$ for all $n\geq 1$ and condition (1) is satisfied.
Also 
\begin{align*}
\nu_2(9^{n-i-1}-9^{m+n-i-1}) = \nu_2(1-9^m)
= 3 + \nu_2(m).
\end{align*} 
For $l=1,2$, this is at least $3$, which is enough. For $l \geq 3$,
let $m\in N_l$ and then we have chosen $N_l$ such that $3 + \nu_2(m) \geq l$. 
So condition (2) is satisfied.
\end{proof}

Next we turn to the periodic versions.
The complex periodic case $K^0(K)_{(p)}$, for $p$ odd,
recovers~\cite[3.2]{ccw3}.

\begin{prop}\label{PropLFGperiodic}
Let $E$ be $K_{(p)}$ for $p$ an odd prime, $G$ or $KO_{(2)}$. Consider the topological
bases for $A=E^0(E)$ given in Theorem~\ref{ThmTBForPeriodicOps}
and let the sets $N_l$ for $l>0$ be as specified in Table~\ref{table}.
Then conditions (1) and (2) of Theorem~\ref{ThmLFGIsDiscrete-specific} are satisfied.
\end{prop}

\begin{proof}
The three theories fit into the framework of Theorem~\ref{ThmLFGIsDiscrete-specific} with
data as follows.

\begin{center}
\begin{tabular}{c|ccc}
$E$&$r$&$\beta$&$z_i$\\
\hline
$K_{(p)}$   &$1$    &$q$    &$q^{(-1)^i\floor{i/2}}$\\
$G$         &$p-1$  &$q$    &$\hat{q}^{(-1)^i\floor{i/2}}$\\
$KO_{(2)}$  &$2$    &$3$    &$9^{(-1)^i\floor{i/2}}$
\end{tabular}
\end{center}

First consider $E=K_{(p)}$ and
write $q_i = q^{(-1)^i \floor{i/2}}$. 
For condition (1) it will be enough to show that $\theta_n(q^j; (q_i))\in p\Zp$ for
$j>0$ and $n\geq p-1$.
Since $q$ generates $(\Z/p)^\times$, the difference $q^r-q_i$ is
divisible by $p$ if and only if $r-(-1)^i \floor{i/2}$ is divisible
by $p-1$. The set $\{ (-1)^i \floor{i/2}\,|\,i = 1,2,\dots,n \}$ 
consists of $n$ consecutive integers, so, since $n \geq p-1$, at least one of the linear factors
of $\theta_n(q^j; (q_i))$ is divisible by $p$.

For condition (2), take $l >0$, $m \in N_l$ and $n \in \N_0$. 
We need to show that $p^l$ divides $q_{n-i}-q_{m+n-i}$
for $0 \leq i \leq n-1$. Since $m$ is even,
$q_{n-i}-q_{m+n-i} = q_{n-i}(1 - q^{(-1)^{n-i}m/2}).$
Since $q$ generates $(\Z/p^l)^\times$ and $p^{l-1}(p-1)$
divides $(-1)^{n-i}m/2$, we deduce that $p^l$ divides $1-q^{(-1)^{n-i}m/2}$, as required.
\smallskip

Next consider $E=G$. Write $\hat{q}_i = \hat{q}^{(-1)^i \floor{i/2}}$.
For each $n \geq 1$ and $j \in \Z$, 
$\theta_n(\hat{q}^j;(q_i))\in p\Zp$, so condition (1) is satisfied.
For condition (2),
take $l >0$ and $m \in N_l$. 
Then 
\[
\nu_p(\hat{q}_{n-i}-\hat{q}_{m+n-i}) = \nu_p(1-\hat{q}^{m/2})
= 1 + \nu_p(m/2)
\geq l,
\]
 as required.
\smallskip

Finally consider $E=KO_{(2)}$ and
write $q_i = 9^{(-1)^i \floor{i/2}}$. 
By \eqref{Eq2adic}, we have $\theta_{n}(9^{j}, (q_i))\in 2\Z_{(2)}$
for every $n \geq 1$, so condition (1) is satisfied.

Take $l >0$ and $m \in N_l$. 
Let 
\begin{align*}
A_{m,n}^i &= 9^{(-1)^{n-i} \floor{(n-i)/2}}-9^{(-1)^{m+n-i} \floor{(m+n-i)/2}}\\
&= 9^{(-1)^{n-i} \floor{(n-i)/2}} (1-9^{m/2}) \qquad\text{since $m$ is even.}
\end{align*}
For condition (2) we need to show that $A_{m,n}^i$ is divisible by $2^l$ for $0\leq i\leq m-1$.
 Now $\nu_2(1-9^{m/2}) = \nu_2(1-3^m) = 2 +\nu_2(m)$, using \eqref{Eq2adic}. For $l=1,2$, this is at least $3$, which is enough; for $l \geq 3$, 
we have chosen $N_l$ such that $2 + \nu_2(m) \geq l$. 
\end{proof}

The last two examples we consider are the $2$-local complex spectra, 
and for these we use Theorem~\ref{ThmLFGIsDiscrete(C)}. 

\begin{prop}
\textup{(1)} Let $A=k^0(k)_{(2)}$. Consider the coalgebra $K_0(k)_{(2)}$ to which 
$A$ is dual, with basis $f_{n}^{(2)}(w)$ as given in Theorem~\ref{basisconnectivecoops}. 
Let $N_l$ be as specified in Table~\ref{table}.
Then conditions (1) and (2) of Theorem~\ref{ThmLFGIsDiscrete(C)} are satisfied.

\textup{(2)} Let $A=K^0(K)_{(2)}$. Consider the coalgebra $K_0(K)_{(2)}$ to which 
$A$ is dual, with basis $F_n^{(2)}(w)$ as
given in Theorem~\ref{periodiccoopsbases}. 
Let $N_l$ be as specified in Table~\ref{table}.
Then conditions (1) and (2) of Theorem~\ref{ThmLFGIsDiscrete(C)} are satisfied.
\end{prop}

\begin{proof}
We only give the details for the first part, since the proof of the second is very similar. 
Note that
\[
wf_{2m}^{(2)}(w) = 3^mf_{2m}^{(2)}(w) - 2.3^mf_{2m+1}^{(2)}(w).
\]
Since $\Delta$ is a map of bialgebras,
\begin{align*}
\Delta(f_{2m+1}^{(2)}(w)) &= \frac{1}{2.3^m}\Delta(3^m-w)\Delta(h_m(w))\\
&= \frac{3^m.1 \tp 1 - w \tp w}{2.3^m}\sum_{i,j}^m \Gamma(KO_0(ko)_{(2)})_{i,j}^m h_i(w) \tp h_j(w),
\end{align*}
so we compute that
\[
\Gamma(K_0(k)_{(2)})_{2i,2j}^{2m+1} = \frac{1-3^{i+j-m}}{2}\Gamma(KO_0(ko)_{(2)})_{i,j}^{m}
\]
and
\[
\Gamma(K_0(k)_{(2)})_{2i,2j+1}^{2m+1} = 3^{i+j-m}\Gamma(KO_0(ko)_{(2)})_{i,j}^{m}.
\]
(We will not need the other cases.)

Write $\eta_n=\theta_n(\Psi^3;9)$ for the basis elements given in Theorem~\ref{ThmTBForOps}. It follows from
Theorem~\ref{ThmLFGIsDiscrete-specific} and Proposition~\ref{PropLFGconnective} that
we have
\begin{enumerate}
\item if $m,n \in N_l$ and $j \in \N_0$, then $\ab{1-\eta_{n-m},w^j} \in 1+2\Z_{(2)}$; and
\item if $m \in N_l$ and $n \in \N_0$, then $\eta_m\eta_n \equiv \eta_{m+n} \mod 2^l$.
\end{enumerate}
The former is equivalent to condition (1) of Theorem~\ref{ThmLFGIsDiscrete(C)} and, since
\[
\eta_m\eta_n = \sum_{i \geq 0} \Gamma(KO_0(ko)_{(2)})_{m,n}^i \eta_i,
\]
the latter is equivalent to condition (2) of Theorem~\ref{ThmLFGIsDiscrete(C)}.
\end{proof}

We have now completed the proof of Theorem \ref{ThmLFGForKs}.

\end{document}